\documentclass[12pt]{article}

\usepackage{authblk}
\usepackage[utf8]{inputenc}
\usepackage{textcomp}

\usepackage{amssymb}
\usepackage{amsmath}
\usepackage{latexsym}
\usepackage{amsthm}

\usepackage{enumerate}
\usepackage{verbatim}
\usepackage{graphicx}

\usepackage{comment}
\usepackage{color}

\allowdisplaybreaks

\makeatletter

\newcommand{\C}{\mathbb C}
\newcommand{\R}{\mathbb R}

\newcommand{\N}{\mathbb N}

\newtheorem{thm}{Theorem}[section]
\newtheorem{lem}[thm]{Lemma}
\newtheorem{prop}[thm]{Proposition}

\newtheorem{exa}[thm]{Example}
\newtheorem{cor}[thm]{Corollary}
\newtheorem{rk}[thm]{Remark}

\numberwithin{equation}{section}

\title{Hyers-Ulam stability of elliptic M\"obius difference equation}

\author[]{Young Woo Nam}

\affil[]{\small Mathematics Section,
      College of Science and Technology,
      Hongik University, 339--701 
      Sejong, Korea} 

\begin{document}

\date{}
\maketitle

\begin{abstract}
The linear fractional map $ f(z) = \frac{az+ b}{cz + d} $ on the Riemann sphere with complex coefficients $ ad-bc \neq 0 $ is called M\"obius map. If $ f $ satisfies $ ad-bc=1 $ and $ -2<a+d<2 $, then $ f $ is called {\em elliptic} M\"obius map. Let $ \{ b_n \}_{n \in \N_0} $ be the solution of the elliptic M\"obius difference equation $ b_{n+1} = f(b_n) $ for every $ n \in \N_0 $. Then the sequence $ \{ b_n \}_{n \in \N_0} $ has no Hyers-Ulam stability. 
\end{abstract}

\section{Introduction}

The first order difference equation is the solution of $ b_{n+1} = F(n,b_n) $ for $ n \in \N_0 $ with the initial point $ b_0 $. For the introductory method and examples see \cite{elaydi}. An interesting non-linear difference equation is the rational difference equation. For instance, Pielou logistic difference equation \cite{pielou} or Beverton-Holt equation \cite{BH,Sen}
are first order rational difference equation as a model for population dynamics with constraint. These equations are understood as the iteration of a kind of M\"obius transformation on the real line. In this paper, we investigate the Hyers-Ulam stability of another kind of M\"obius transformation which does not appear in population dynamics and extend the result to the complex plane. 
\medskip \\
Hyers-Ulam stability raised from Ulam's question \cite{ulam} about the stability of approximate homomorphism between metric groups. The first answer to this question was given by Hyers \cite{hyers}
for Cauchy additive equation in Banach space. Later, the theory of Hyers-Ulam stability is developed in the area of functional equation and differential equation by many authors. The theory of Hyers-Ulam stability for difference equation appears in relatively recent decades and is mainly searched for linear difference equations, for example, see \cite{jung1,jungnam,popa, XB} 
. Denote the set of natural numbers by $ \N $ and denote the set $ \N \cup \{\infty\} $ by $ \N_0 $. The set of real numbers and complex numbers by $ \R $ and $ \C $ respectively. Denote the unit circle by $ \mathbb{S}^1 $. 
\medskip \\
Suppose that the complex valued sequence $ \{ a_n\}_{n \in \N} $ satisfies the inequality
\begin{align*}
| a_{n+1} - F(n,a_n)| \leq \varepsilon
\end{align*}
for a $ \varepsilon > 0 $ and for all $ n \in \N_0 $, where $ | \cdot | $ is the absolute value of complex number. If there exists a sequence $ \{ b_n\}_{n \in \N} $ which satisfies that
\begin{align} \label{eq-Hyers Ulam stability}
b_{n+1} = F(n,b_n) 
\end{align}
for each $ n \in \N_0 $ and $ | a_n - b_n| \leq G(\varepsilon) $ for all $ n \in \N_0 $, where the positive number $ G(\varepsilon) \rightarrow 0 $ as $ \varepsilon \rightarrow 0 $. Then we say that the difference equation \eqref{eq-Hyers Ulam stability} has {\em Hyers-Ulam stability}.

\subsection*{Classification of M\"obius transformation}

Denote the Riemann sphere by $ \hat{\C} $, which is the one point compactification of the complex plane, namely, $ \C \cup \{ \infty \} $. Similarly, we define the {\em extended real line} as $ \R \cup \{ \infty \} $ and denote it by $ \hat{\R} $. 
M\"obius transformation (or M\"obius map) is the linear fractional map defined on $ \hat{\C} $ as follows
\begin{align} \label{eq-Mobius transformation}
g(z) = \frac{az + b}{cz + d}
\end{align}
where $ a,b,c $ and $ d $ are complex numbers and $ ad-bc \neq 0 $. Define $ g\left( -\frac{d}{c} \right) = \infty $ and $ g(\infty) = \frac{a}{c} $. 
If $ c = 0 $, then $ g $ is the linear function. Thus we assume that $ c \neq 0 $ throughout this paper. The M\"obius map which preserves $ \hat{\R} $ is called the {\em real} M\"obius map. A M\"obius map is real if and only if the coefficients of the map $ a,b,c $ and $ d $ are real numbers. 
\medskip \\
The M\"obius map has two fixed points counting with multiplicity. Denote these points by $ \alpha $ and $ \beta $. The real M\"obius maps are classified to the three different cases using fixed points. 
\begin{itemize}
\item If $ \alpha $ and $ \beta $ are real distinct numbers, the map is called real {\em hyperbolic} M\"obius map, \vspace{-5pt}
\item If $ \alpha = \beta $, then the map is called real {\em parabolic} M\"obius map, and \vspace{-5pt}
\item If $ \alpha $ and $ \beta $ are two distinct non-real complex numbers, then the map is called real {\em elliptic} M\"obius map.
\end{itemize}
\smallskip
M\"obius map $ x \mapsto \frac{ax + b}{cx + d} $ is the same as $ x \mapsto \frac{pax + pb}{pcx + pd} $ for all numbers $ p \neq 0 $. Thus we may assume that $ ad - bc = 1 $ when we choose $ p = \sqrt{ad - bc} $. Moreover, M\"obius map has the matrix representation 
$ \big(\begin{smallmatrix} a & b \\ c & d \end{smallmatrix} \big) $ under the condition $ ad - bc =1 $. Denote the matrix representation of the M\"obius map $ g $, by also $ g $ and its trace by $ \mathrm{tr}(g) $, which means $ a + d $. In the complex analysis or hyperbolic geometry, M\"obius maps with complex coefficients can be classified similarly with different method. For instance, see \cite{beardon}. 
M\"obius transformation in \eqref{eq-Mobius transformation} (with real or complex coefficients) for $ ad-bc =1 $ is classified as follows 
\begin{itemize}
\item If $ \mathrm{tr}(g) \in \R \setminus [-2,2] $, then $ g $ is called {\em hyperbolic},
\item If $ \mathrm{tr}(g) = \pm 2 $, then $ g $ is called {\em parabolic},
\item If $ \mathrm{tr}(g) \in (-2,2) $, then $ g $ is called {\em elliptic} and 
\item If $ \mathrm{tr}(g) \in \C \setminus \R $, then $ g $ is called {\em purely loxodromic}. 
\end{itemize}
The real M\"obius maps are also classified by the above notions. In this paper, we investigate Hyers-Ulam stability of elliptic M\"obius transformations. Other cases would appear in the forthcoming papers.

\section{No Hyers-Ulam stability with dense subset}

In this section we prove non-stability in the sense of Hyers-Ulam, which is not only for the elliptic M\"obius transformation but also for any function satisfying the assumption of the following Theorem \ref{thm-no hyers ulam stability}.

\begin{thm} \label{thm-no hyers ulam stability}
Let $ \{ b_n \}_{n\in \N_0} $ be the sequence in $ \R $ satisfying $ b_{n+1} = F(b_n) $ with a map $ F $ for $ n \in \N_0 $. Suppose that there exists a dense subset $ A $
of $ \R $ such that if $ b_0 \in A $, then the sequence $ \{ b_n \}_{n\in \N_0} $ is dense in $ \R $. Suppose also that $ \{ b_n \}_{n\in \N_0} $ has no periodic point. Then the sequence $ \{ b_n \}_{n\in \N_0} $ has no Hyers-Ulam stability. 
\end{thm}

\begin{proof}
For any $ a_0 \in \R $ and $ \varepsilon > 0 $, choose the sequence $ \{ a_n \}_{n\in \N_0} $ as follows
\begin{enumerate} 
\item $ a_0 \in \R $ is arbitrary, \vspace{-5pt}
\item $ a_1 $ satisfies that $ | F(a_0) - a_1| \leq \varepsilon $ and $ a_1 \in A $, that is, the sequence $ \{ F^n(a_1) \}_{n\in \N_0} $ is dense in $ \R $. 
\end{enumerate}
Let $ n_{k} $ for $ k \geq 1 $ be the positive numbers $ n_1 < n_2 < \cdots < n_k < \cdots $ such that $ F^{n_k}(a_1) $ are points in the ball of which center is $ a_1 $ and diameter is $ \varepsilon $. 
\begin{enumerate}
\item[3.] $ a_{n+1} = F^{n}(a_1) $ for $ n= 0,1,2, \ldots ,n_1 -1$,  \vspace{-5pt} 
\item[4.] $ a_{n_1+1} = a_1 $ \;and \;$ a_k = a_{k + n_1+1} $ for every $ k \in \N $. 
\end{enumerate}
Then the sequence $ \{ a_n \}_{n\in \N_0} $ satisfies that 
$ |a_{n+1} - F(a_n) | \leq \varepsilon $ for all $ n \in \N_0 $. Moreover, since the sequence $ \{ a_n \}_{n\in \N} $ is periodic, $ \{ a_n \}_{n\in \N_0} $  is the finite set and it is bounded. However, the fact that the sequence $ \{ b_n \}_{n\in \N_0} $ is dense in $ \R $ implies that $ | a_n - b_n | $ is unbounded for $ n \in \N $.
Hence, the sequence $ \{ b_n \}_{n\in \N_0} $ does not have Hyers-Ulam stability. 
\end{proof}
\medskip
\begin{rk}
Theorem \ref{thm-no hyers ulam stability} can be generalized to any metric space only if the definition of Hyers-Ulam stability is modified suitably. For instance, if $ F $ is the map from the metric space $ X $ to itself and $ | \cdot | $ is changed to the distance $ \mathrm{dist}( \cdot \,, \cdot) $ from the metric on $ X $, then we can define Hyers-Ulam stability on the metric space and Theorem \ref{thm-no hyers ulam stability} is applied to it. For example, the unit circle $ \mathbb{S}^1 $ is the metric space of which distance between two points defined from the minimal arc length connecting these two points. Then Hyers-Ulam stability on $ \mathbb{S}^1 $ can be defined. 
\end{rk}

\section{Real elliptic M\"obius transformation}  \label{sec-real elliptic map}

\begin{lem} \label{lem-real elliptic Mob transform}
Let $ g(x) = \frac{ax +b}{ cx +d} $ be the linear fractional map where $ a, b, c $ and $ d $ are real numbers, $ c \neq 0 $ and $ ad-bc =1 $. Then $ g $ has fixed points which are non-real complex numbers if and only if $ g $ is real elliptic M\"obius map, that is, $ -2 < a + d < 2 $. 
\end{lem}

\begin{proof}
The equation $ g(x) = x $ implies that 
$$ x = \frac{a-d \pm \sqrt{(a+d)^2 - 4}}{2c} . $$
Then the fixed points are non-real complex numbers if and only if $ -2 < a + d < 2 $. 
\end{proof}
%

\begin{lem} \label{lem-conjugation and rotation}
Let $ g $ be the map defined on $ \hat{\C} $ in Lemma \ref{lem-real elliptic Mob transform} and two complex numbers $ \alpha $ and its complex conjugate $ \bar{\alpha} $ be the fixed points of $ g $. Let $ h $ be the map defined as $ h(x) = \frac{x - \alpha}{x - \bar{\alpha}} $. If $ g $ is the elliptic M\"obius map, that is, $ -2 < a + d < 2 $, then 
\begin{align*}
h \circ g \circ h^{-1} (x) = \frac{x}{(c\alpha + d)^2} .
\end{align*} 
for $ x \in \hat{\C} $. Moreover, $ | g'(\alpha)| = | c\alpha + d| = 1 $. 
\end{lem}

\begin{proof}
The map $ h \circ g \circ h^{-1} $ has the fixed points $ 0 $ and $ \infty $. Since both $ g $ and $ h $ are linear fractional maps, so is $ h \circ g \circ h^{-1} $. Then $ h \circ g \circ h^{-1} (x) = kx $ for some $ k \in \C $. The equation $ h \circ g(x) = kh(x) $ implies that $ h'(g(x))g'(x) = kh'(x) $. Thus 
\begin{align*}
h'(g(\alpha))g'(\alpha) = h'(\alpha)g'(\alpha) = kh'(\alpha)
\end{align*}
Then $ k = g'(\alpha)=\frac{1}{(c\alpha + d)^2} $. Moreover, $ |g'(\alpha)| = 1 $ if and only if $ |c\alpha + d| = 1 $. Recall that both $ \alpha $ and $ \bar{\alpha} $ are roots of the equation, $ cx^2 - (a-d)x -b = 0 $. Then
\begin{align*}
|c\alpha + d|^2 &= (c\alpha + d)(c\bar{\alpha} + d)  \\
&= c^2 \alpha \bar{\alpha} + cd (\alpha + \bar{\alpha}) +d^2 \\
&= c^2 \left(- \frac{b}{c} \right) + cd \,\frac{a-d}{c} + d^2 \\
&= -bc + ad - d^2 + d^2 \\
&= ad-bc \\
&= 1 .
\end{align*}
Hence, $ | g'(\alpha)| = \frac{1}{|c\alpha + d|^2} = 1 $. 
\end{proof}
\smallskip
By Lemma \ref{lem-conjugation and rotation}, the map $ h \circ g \circ h^{-1} $ is a rotation on $ \mathbb{S}^1 $. Since $ h^{-1} $ is bijective from $ \mathbb{S}^1 \setminus \{1\} $ to $ \R $, $ x $ is a periodic point under $ h \circ g \circ h^{-1} $ in $ \mathbb{S}^1 \setminus \{1\} $ with period $ p $ if and only if $ h^{-1}(x) $ is periodic in $ \R $ with the same period. 
Recall that $ h^{-1}(1) = \infty $. Thus when we investigate Hyers-Ulam stability of the sequence $ \{ b_n \}_{n \in \N_0} $ as the solution of the elliptic linear fractional map $ g $, we have to choose carefully the initial point $ b_0 \in \R $ satisfying $ g^k(b_0) \neq \infty $ for all $ k \in \N $. 

\medskip
\begin{prop} \label{prop-density of elliptic Mobius acc points}
Let $ g $ be the elliptic linear fractional map defined in Lemma \ref{lem-real elliptic Mob transform} on $ \hat{\R} $. Suppose that there exists $ x \in \R $ such that $ g^k(x) \neq x $ for all $ k \in \N $. Then the sequence $ \{g^k(x) \}_{k \in \N} $ is dense in $ \R $ where $ x \in \R \setminus \{ g^{-k}(\infty) \}_{k \in \N} $. 
\end{prop}

\begin{proof}
Lemma \ref{lem-conjugation and rotation} implies that $ h \circ g \circ h^{-1} (x) = e^{i\theta} x $ for some $ \theta \in \R $. If $ \theta $ is a rational number $ \frac{q}{p} $, then $ h \circ g^p \circ h^{-1} (x) = x $ for all $ x \in \C $. Thus $ g^p(x) = x $ for all $ x \in \C $. Then $ \theta $ is an irrational number. Since $ x \mapsto e^{i\theta} x $ is an irrational rotation on $ \mathbb{S}^1 $, the sequence $ \{ e^{ik\theta}x \}_{k \in \N } $ is dense in $ \mathbb{S}^1 $ for every $ x \in \mathbb{S}^1 $. Moreover, when $ x=1 $ is chosen, the set $ \mathbb{S}^1 \setminus \{ e^{-ik\theta} \}_{k \in \N_0 } $ is also a dense subset of $ \mathbb{S}^1 $.
\smallskip \\
The direct calculation implies that $ h^{-1}(x) = \frac{\bar{\alpha}x - \alpha}{x-1} $ and then $ h^{-1}(1) = \infty $. Choose two points $ p $ and $ p' $ close enough to each other in $ \mathbb{S}^1 \setminus \{ e^{-ik\theta} \}_{k \in \N_0 } $. Then
\begin{align} \label{eq-distance between points 1}
| h^{-1}(p) - h^{-1}(p') | &= \left| \frac{\bar{\alpha}p - \alpha}{p-1} - \frac{\bar{\alpha}p' - \alpha}{p'-1} \right| \nonumber \\[0.2em] 
&= \left| \frac{(\bar{\alpha}p - \alpha)(p'-1) - (\bar{\alpha}p' - \alpha)(p-1)}{(p-1)(p'-1)} \right|  \nonumber \\[0.2em] 
&= \left| \frac{(\alpha - \bar{\alpha})(p-p')}{(p-1)(p'-1)} \right|   \nonumber \\[0.2em] 
&= \left| \frac{\alpha - \bar{\alpha}}{(p-1)(p'-1)} \right| \cdot |p - p'|  .
\end{align}
We choose the sequence $ \{ p_{n_k}\}_{k \in \N_0 } $ for some $ p_0 \in \mathbb{S}^1 \setminus \{ e^{-ik\theta} \}_{k \in \N_0 } $ which satisfies that
$ h \circ g^{n_k} \circ h^{-1}(p_0) = p_{n_k} $ and $ p_{n_k} \rightarrow p $ as $ k \rightarrow \infty $ for different numbers $ n_1 < n_2 < \cdots < n_k < \cdots $. Since $ h^{-1} $ is a bijection from $ \mathbb{S}^1 \setminus \{ 1\} $ to $ \R $, by the equation \eqref{eq-distance between points 1} we obtain that 
\begin{align*}
| h^{-1}(p) - g^{n_k} \circ h^{-1}(p_0)| &= | h^{-1}(p) - h^{-1}(p_{n_k}) | \\[0.2em]
& \leq \left( \left| \frac{\alpha - \bar{\alpha}}{(p-1)^2}\right|  + \delta \right) |p - p_{n_k}| 
\end{align*}
for some $ \delta > 0 $. Then any point $ h^{-1}(p) $ in $ \R $ is an accumulation point in the sequence $ \{ g^k(x)  \}_{ k \in \N_0} $ where $ x \in h^{-1} \big(\mathbb{S}^1 \setminus \{ e^{-ik\theta} \}_{k \in \N_0 } \big) $, which is a dense subset of $ \R $. 
\end{proof}

Theorem \ref{thm-no hyers ulam stability} and Proposition \ref{prop-density of elliptic Mobius acc points} imply that the real elliptic linear fractional map does not have Hyers-Ulam stability. 

\smallskip

\begin{cor}  \label{cor-no stability of irrational rot1}
Let $ g(x) = \frac{ax +b}{cx +d} $ be the linear fractional map on $ \hat{\R} $ where $ a, b, c $ and $ d $ are real numbers, $ c \neq 0 $ and $ ad-bc =1 $. Suppose that $ -2 < a + d < 2 $ and $ h \circ g \circ h^{-1} $ is an irrational rotation on the unit circle where $ h(x) = \frac{x - \alpha}{x - \bar{\alpha}} $. If the sequence $ \{ b_n \}_{n \in \N_0} $ in $ \R $ is the solution of $ b_{n+1} = g(b_n) $ for $ n \in \N_0 $, then either $ g^k(b_0) = \infty $ for some $ k \in \N $ or $ \{ b_n \}_{n \in \N_0} $ has no Hyers-Ulam stability . 
\end{cor}

\section{Extension of non stability to complex plane}

In this section, we extend no Hyers-Ulam stability of the real elliptic linear fractional map to the elliptic M\"obius transformation with complex coefficients. 
Let $ \ell $ be the straight line in the complex plane. Define the extended line as $ \ell \cup \{ \infty \} $ and denote it by $ \hat{\ell} $. Interior of the circle $ C $ in $ \C $ means that the bounded region of the set $ \C \setminus C $. 

\begin{lem} \label{lem-elliptic Mob transform}
Let $ f(z) = \frac{az + b}{cz + d} $ be the M\"obius map with complex coefficients $ a,b,c $ and $ d $ for $ c \neq 0 $ and $ ad - bc = 1 $. If $ f $ is elliptic, that is, $ -2 < a + d < 2 $, then there exists the unique extended line which is invariant under $ g $ in $ \hat{\C} $. 
\end{lem}

\begin{proof}
Let $ \alpha $ and $ \beta $ be the fixed points of $ f $. In particular, the fixed points of $ g $ are as follows 
\begin{align*}
\alpha = \frac{a-d + \sqrt{(a+d)^2 - 4}}{2c}, \quad \beta = \frac{a-d - \sqrt{(a+d)^2 - 4}}{2c}
\end{align*}
Denote the straight line, \smallskip $ \{ z \in \C : |z - \alpha | = |z - \beta | \} $ by $ \ell $ in $ \C $ and denote the extended line $ \ell \cup \{ \infty \} $ by $ \hat{\ell} $. 
We prove that 
$ \hat{\ell} 
 $ is the unique invariant extended line under $ f $.
\medskip \\
{\em Claim} : The points, $ -\frac{d}{c} $ and $ \frac{a}{c} $ are in $ \hat{\ell} $. 
\begin{align*}
\alpha + \frac{d}{c} = \frac{a + d + \sqrt{(a+d)^2 - 4}}{2c}, \quad \beta + \frac{d}{c} = \frac{a + d - \sqrt{(a+d)^2 - 4}}{2c}
\end{align*}
The fact that $ a+d $ is a real number and $ \sqrt{(a+d)^2 - 4} $ is a purely imaginary number implies that 
\begin{align} \label{eq-same absolute value of some points}
\left| \alpha + \frac{d}{c} \right | = \left| \beta + \frac{d}{c} \right | = \frac{1}{|c|} .
\end{align}
Then $ -\frac{d}{c} $ is in $ \hat{\ell} $. By similar calculation, $ \frac{a}{c} $ is also in $ \hat{\ell} $. The proof the claim is complete. 
\medskip \\
For the invariance of $ \hat{\ell} $, it suffice to show that if $ z \in {\ell} \setminus \{-\frac{d}{c} \} $, then $ f(z) \in \ell $. For any $ z \in {\ell} \setminus \{-\frac{d}{c} \} $, we have that
\begin{align} \label{eq-invariance of extended line 1}
f(z) - \alpha &= f(z) - f(\alpha) \nonumber \\
&= \frac{az + b}{cz + d} - \frac{a\alpha + b}{c\alpha + d} 
= \frac{(ad-bc)(z- \alpha)}{(cz + d)(c\alpha + d)} 
= \frac{z- \alpha}{(cz + d)(c\alpha + d)}
\end{align}
By similar calculation, we obtain that 
\begin{align} \label{eq-invariance of extended line 2}
f(z) - \beta = \frac{z- \beta}{(cz + d)(c\beta + d)} .
\end{align}
The equation \eqref{eq-same absolute value of some points} in the claim implies that $ | c\alpha + d| = | c\beta + d| $. The definition of $ \ell $ implies that $ cz + d \neq 0 $ for $ z \in {\ell} \setminus \{-\frac{d}{c} \} $ and $ |z - \alpha | = |z - \beta | $. The equation $ |f(z) - \alpha| = |f(z) - \beta| $ holds by comparing the equation \eqref{eq-invariance of extended line 1} with \eqref{eq-invariance of extended line 2}. Hence, by this result and the above claim, the extended line $ \hat{\ell} $ is invariant under $ f $. There is the unique straight line which connects $ -\frac{d}{c} $ and $ \frac{a}{c} $ in $ \C $. Hence, $ \hat{\ell} $ is the unique invariant extended line in $ \hat{\C} $ under $ f $. 
\end{proof}

\smallskip
\begin{rk} \label{rk-concentric circles}
Define the map $ h $ as $ h(z) = \frac{z - \beta}{z - \alpha} $. Then by the straightforward calculation we obtain $ h \circ f \circ h^{-1}(z) = f'(\alpha) z $ and $ |f'(\alpha)| = 1 $. Let $ C_r $ be the circle of radius $ r>0 $ of which center is the origin in $ \C $. Observe that a fixed point of $ f $ is contained in the interior of $ h^{-1}(C_r) $ for all $ r>0 $. If $ h \circ g \circ h^{-1} $ is an irrational rotation in $ \C $, then every concentric circles $ C_r $ are invariant under $ h \circ f \circ h^{-1} $. Then $ h^{-1}(C_r) $ is an invariant circle under $ f $ for every $ r > 0 $. However, since $ h(\infty) = 1 $, the unique invariant extended line under $ f $ is $ h^{-1}(\mathbb{S}^1) $. Moreover, $ h^{-1}(C_{1/r}) $ is the reflected image of $ h^{-1}(C_r) $ to the invariant extended line. 
\end{rk}

Due to the existence of the extended line which is invariant under $ g $, no Hyers-Ulam stability of the sequence $ \{ g^n(z) \}_{n \in \N_0} $ of the {\em real} elliptic linear fractional map is extendible to that of the elliptic M\"obius transformation.

\begin{prop} \label{prop-density of elliptic Mobius acc points 2}
Let $ f $ be the elliptic M\"obius transformation on $ \hat{\C} $ and $ \hat{\ell} $ be the extended line invariant under $ f $. Suppose that there exists $ z \in \hat{\ell} $  such that $ f^k(z) \neq z $ for all $ k \in \N $. Then the sequence $ \{f^k(z) \}_{k \in \N} $ is dense in $ \ell = \hat{\ell} \setminus \{ \infty\} $ where $ z \in \ell \setminus \{ f^{-k}(\infty) \}_{k \in \N} $. 
\end{prop}
\begin{proof}
Lemma \ref{lem-elliptic Mob transform} implies that there exists the invariant extended line under $ f $. In the proof of Proposition \ref{prop-density of elliptic Mobius acc points}, replace $ \bar{\alpha} $ by $ \beta $ and apply the proof of Proposition \ref{prop-density of elliptic Mobius acc points}. Then the similar calculation completes the proof. 
\end{proof}

Then Proposition \ref{prop-density of elliptic Mobius acc points 2} and Theorem \ref{thm-no hyers ulam stability} implies that no Hyers-Ulam stability of $ \{ f^n(z) \}_{n \in \N_0} $ where $ f $ is the elliptic M\"obius map for $ z \in \C $.

\begin{cor} \label{cor-no stability of irrational rot2}
Let $ f(z) = \frac{az +b}{cz +d} $ be the M\"obius map on $ \hat{\C} $ where $ c \neq 0 $ and $ ad-bc =1 $. If $ f $ is elliptic, that is, $ -2 < a + d < 2 $, then the sequence $ \{ b_n \}_{n \in \N_0} $ satisfying $ b_{n+1} = f(b_n) $ for $ n \in \N_0 $ has no Hyers-Ulam stability on $ \C $. 
\end{cor}
\begin{proof}
The sequence $ \{ a_n \}_{n \in \N_0} $ is in the invariant extended line $ \hat{\ell} $ under $ f $ satisfying 
\begin{align*}
| a_{n+1} - f(a_n) | \leq \varepsilon
\end{align*}
for $ n \in \N_0 $. Choose $ \{ a_n \}_{n \in \N_0} $ is the finite sequence which satisfies the properties in the proof of Theorem \ref{thm-no hyers ulam stability} in $ \ell $. 
Then the similar proof in Section \ref{sec-real elliptic map} implies that any sequence $ \{ b_n \}_{n \in \N_0} $ satisfying $ b_{n+1} = f(b_n) $ for every $ n \in \N_0 $ in $ \ell $ has no Hyers-Ulam stability. 
\medskip\\
Let $ h $ be the map $ h(z) = \frac{z - \alpha}{z - \beta} $. Observe that $ h \circ f \circ h^{-1} $ is an irrational rotation on $ \C $. If $ \{ b_n \}_{n \in \N_0} $ is contained in $ \C \setminus \hat{\ell} $, then $ \{ b_n \}_{n \in \N_0} $ is contained in the circle disjoint from $ \ell $ and this circle is $ h^{-1}(C_r) $ for some $ r>0 $. Observe that if $ r = 0 $, then $ C_r = \{ \alpha \} $ and if $ r =1 $, then $ C_r = \hat{\ell} $. Since $ h \circ f \circ h^{-1} $ is an irrational rotation, $ \{ f^n(h^{-1}(x)) \}_{n \in \N_0} $ is the dense subset of $ h^{-1}(C_r) $. We may assume that $ 0 < r < 1 $ and the fixed point $ \alpha $ is contained in the interior of $ C_r $ by Remark \ref{rk-concentric circles}. Denote the (minimal) distance between the point $ \alpha $ and the line $ {\ell} $ by $ L $. Then the density of $ \{ b_n \}_{n \in \N_0} $ on the circle $ C_r $ and the finiteness of $ \{ a_n \}_{n \in \N_0} $ imply that $ | a_n - b_n | \geq L $ for infinitely many $ n \in \N $ for all $ \varepsilon < L $. Hence, $ \{ b_n \}_{n \in \N_0} $ has no Hyers-Ulam stability. 
\end{proof}

\section{Non stability of periodic sequence}

Let the sequence $ \{ b_n \}_{n \in \N_0} $ satisfying $ b_{n+p} = b_n $ for every $ n \in \N_0 $ for some $ p \in \N $ be {\em periodic sequence}. The least positive number $ p $ satisfying the above equation is called the {\em peroid} of sequence. If $ p =1 $, then it is called {\em constant sequence}. 

\begin{lem} \label{lem-periodic sequence no stability}
Let the sequence $ \{ b_n \}_{n \in \N_0} $ in $ \R $ be a periodic sequence with period $ p $. Then $ \{ b_n \}_{n \in \N_0} $ has no Hyers-Ulam stability. 
\end{lem}

\begin{proof}
Any periodic sequence has constant subsequence. For example, let $ \{ c_n \}_{n \in \N_0} $ be the sequence satisfying $ c_n = b_{pn} $ for every $ n \in \N_0 $. Thus $ \{ c_n \}_{n \in \N_0} $ is the constant sequence $ \{ c_0 \}_{n \in \N_0} $. It suffice to show that the constant sequence has no Hyers-Ulam stability. For any small enough $ \varepsilon > 0 $, define the sequence $ \{ d_n \}_{n \in \N_0} $ as follows
\begin{enumerate}
\item $ d_0 $ is arbitrary and  \vspace{-5pt}
\item $ d_n = d_0 + n\varepsilon $ for $ n \in \N $
\end{enumerate}
Then $ | d_{n+1} - d_n | \leq \varepsilon $ for all $ n \in \N_0 $. 
However, for any constant sequence $ \{ c_0 \}_{n \in \N_0} $ satisfying $ | c_0 -d_0 | \leq \varepsilon $ such that 
\begin{align*}
| c_k - d_k | = | c_0 - d_k | = | c_0 - d_0 - n\varepsilon | \geq -| c_0 - d_0 | + n\varepsilon \geq (n-1)\varepsilon
\end{align*}
for all $ n \geq 2 $. Since $ | c_k - d_k | $ is unbounded for $ n \in \N $, the constant sequence $ \{ c_n \}_{n \in \N_0} $ has no Hyers-Ulam stability. Hence, the periodic sequence $ \{ b_n \}_{n \in \N_0} $  has no Hyers-Ulam stability either. 
\end{proof}

\smallskip
\begin{exa}
There are non-linear M\"obius transformations, of which  finitely many composition is the identity map. For example, see the following maps 
\begin{align*}
p(z) = \frac{\sqrt{3}z - 2}{2z - \sqrt{3}}, \quad q(z) = \frac{-1}{z - \sqrt{3}}, \quad r(z) = \frac{-z-1}{z} .
\end{align*}
Thus $ p \circ p(z) = p^2(z) = z $ for all $ z \in \C $ and the trace of $ p $ is that $ \mathrm{tr}(p) = \sqrt{3} + (-\sqrt{3}) = 0 $. The map $ q $ satisfies that $ q^3 = p $ by the direct calculation and then $ q^6(z) = z $ for all $ z \in \C $. The trace of $ q $ is that $ \mathrm{tr}(q) = 0 + (-\sqrt{3}) = -\sqrt{3} $. Finally, $ r^3(z) = z $  for all $ z \in \C $ and $ \mathrm{tr}(r) = -1 + 0 = -1 $. All of the traces of $ p $, $ q $ and $ r $ are between $ -2 $ and $ 2 $. 
\end{exa}

\medskip
Remark \ref{rk-concentric circles} implies that $ h \circ g \circ h^{-1}(x) = e^{i\theta}x $ for every elliptic M\"obius map $ g $. If $ \theta = \frac{q}{p} $, then $ g^p(x) =x $ for all $ x \in \C $. Thus the sequence $ \{ g^n(x) \}_{n \in \N_0} $ is periodic. Then Corollary \ref{cor-no stability of irrational rot1} and Lemma \ref{lem-periodic sequence no stability} implies the following theorem.

\medskip
\begin{thm}
Let $ g(x) = \frac{ax + b}{cx + d} $ be the linear fractional map on $ \hat{\R} $ for $ c \neq 0 $, $ ad-bc = 1 $. Supposet that $ g $ is the elliptic linear fractional map, that is, $ -2 <a +d < 2 $. Then the sequence $ \{ b_n \}_{n \in \N_0} $ in $ \R $ satisfying $ b_{n+1} = g(b_n) $ for $ n \in \N_0 $ either satisfies that $ g^k(b_0) = \infty $ for some $ k \in \N $ or it has no Hyers-Ulam stabiliy. 
\end{thm}

If the sequence $ \{ f^n(z) \}_{n \in \N_0} $ is periodic, then Corollary \ref{cor-no stability of irrational rot2} and Lemma \ref{lem-periodic sequence no stability} implies the following theorem.

\begin{thm}
Let $ f(z) = \frac{az + b}{cz + d} $ be the M\"obius map on $ \hat{\C} $ for $ ad-bc = 1 $, $ c \neq 0 $. Suppose that $ f $ is the elliptic M\"obius map, that is, $ -2 <a +d < 2 $. Then the sequence $ \{ b_n \}_{n \in \N_0} $ in $ \C $ satisfying $ b_{n+1} = f(b_n) $ for $ n \in \N_0 $ either satisfies that $ g^k(b_0) = \infty $ for some $ k \in \N $ or it has no Hyers-Ulam stabiliy. 
\end{thm}



\begin{thebibliography}{16}
\ifx \bisbn      \undefined \def \bisbn  #1{ISBN #1}\fi
\ifx \binits     \undefined \def \binits#1{#1}\fi
\ifx \bauthor    \undefined \def \bauthor#1{#1}\fi
\ifx \batitle    \undefined \def \batitle#1{#1}\fi
\ifx \bjtitle    \undefined \def \bjtitle#1{#1}\fi
\ifx \bvolume    \undefined \def \bvolume#1{\textbf{#1}}\fi
\ifx \byear      \undefined \def \byear#1{#1}\fi
\ifx \bissue     \undefined \def \bissue#1{#1}\fi
\ifx \bfpage     \undefined \def \bfpage#1{#1}\fi
\ifx \blpage     \undefined \def \blpage #1{#1}\fi
\ifx \burl       \undefined \def \burl#1{\textsf{#1}}\fi
\ifx \doiurl     \undefined \def \doiurl#1{\textsf{#1}}\fi
\ifx \betal      \undefined \def \betal{\textit{et al.}}\fi
\ifx \binstitute \undefined \def \binstitute#1{#1}\fi
\ifx \binstitutionaled      \undefined \def \binstitutionaled#1{#1}\fi
\ifx \bctitle    \undefined \def \bctitle#1{#1}\fi
\ifx \beditor    \undefined \def \beditor#1{#1}\fi
\ifx \bpublisher \undefined \def \bpublisher#1{#1}\fi
\ifx \bbtitle    \undefined \def \bbtitle#1{#1}\fi
\ifx \bedition   \undefined \def \bedition#1{#1}\fi
\ifx \bseriesno  \undefined \def \bseriesno#1{#1}\fi
\ifx \blocation  \undefined \def \blocation#1{#1}\fi
\ifx \bsertitle  \undefined \def \bsertitle#1{#1}\fi
\ifx \bsnm       \undefined \def \bsnm#1{#1}\fi
\ifx \bsuffix    \undefined \def \bsuffix#1{#1}\fi
\ifx \bparticle  \undefined \def \bparticle#1{#1}\fi
\ifx \barticle   \undefined \def \barticle#1{#1}\fi
\ifx \bconfdate  \undefined \def \bconfdate #1{#1}\fi
\ifx \botherref  \undefined \def \botherref #1{#1}\fi
\ifx \url        \undefined \def \url#1{\textsf{#1}}\fi
\ifx \bchapter   \undefined \def \bchapter#1{#1}\fi
\ifx \bbook      \undefined \def \bbook#1{#1}\fi
\ifx \bcomment   \undefined \def \bcomment#1{#1}\fi
\ifx \oauthor    \undefined \def \oauthor#1{#1}\fi
\ifx \citeauthoryear        \undefined \def \citeauthoryear#1{#1}\fi
\ifx \endbibitem \undefined \def \endbibitem {}\fi
\ifx \bconflocation         \undefined \def \bconflocation#1{#1}\fi
\ifx \arxivurl   \undefined \def \arxivurl#1{\textsf{#1}}\fi
\csname PreBibitemsHook\endcsname

\bibitem{beardon}
\begin{barticle}
\bauthor{\bsnm{Beardon}, \binits{A F}}:
\bbook {The geometry of discrete groups}.
\bpublisher{Springer-Verlag},
\bseriesno{Graduate Texts in Mathematics 91}
(\byear{1983})
\end{barticle}
\endbibitem


\bibitem{BH}
\begin{barticle}
\bauthor{\bsnm{Bohner}, \binits{M}},
\bauthor{\bsnm{Warth}, \binits{H}}:
\batitle{The Beverton-Holt dynamic equation}.
\bjtitle{Applicable Analysis}
\bvolume{86}, 
\bfpage{1007}--\blpage{1015}
(\byear{2007}) 
\end{barticle}
\endbibitem










\bibitem{elaydi}
\begin{barticle}
\bauthor{\bsnm{Elaydi}, \binits{S N}}:
\bbook {An Introduction to Difference Equations}.
\bpublisher{Springer}
(\byear{2005})
\end{barticle}
\endbibitem



\bibitem{hyers}
\begin{barticle}
\bauthor{\bsnm{Hyers}, \binits{D H}}:
\batitle{On the stability of the linear functional equation}.
\bjtitle{Proc. Natl. Acad. Sci. USA}
\bvolume{27},
\bfpage{222}--\blpage{224}
(\byear{1941})
\end{barticle}
\endbibitem





\bibitem{jung1}
\begin{barticle}
\bauthor{\bsnm{Jung}, \binits{S-M}}:
\batitle{Hyers-Ulam stability of the first-order matrix
         difference equations}.
\bjtitle{Adv. Difference Equ.}
\bfpage{no. 170}
(\byear{2015})
\end{barticle}
\endbibitem


\bibitem{jungnam}
\begin{barticle}
\bauthor{\bsnm{Jung}, \binits{S-M}},
\bauthor{\bsnm{Nam}, \binits{Y W}}:
\batitle{On the Hyers-Ulam stability of the first order  difference equation}.
\bjtitle{J. Function Spaces}, 
\barticle{Article ID 6078298}
(\byear{2016})
\end{barticle}
\endbibitem




\bibitem{pielou}
\begin{barticle}
\bauthor{\bsnm{Pielou}, \binits{E C}}:
\batitle{Population and community ecology}.
\bjtitle{Gordon and Breach, Science publishers}
(\byear{1974})
\end{barticle}
\endbibitem


\bibitem{popa}
\begin{barticle}
\bauthor{\bsnm{Popa}, \binits{D}}:
\batitle{Hyers-Ulam-Rassias stability of a linear recurrence}.
\bjtitle{J. Math. Anal. Appl.}
\bvolume{309}, 
\bfpage{591}--\blpage{597}
(\byear{2015})
\end{barticle}
\endbibitem




\bibitem{Sen}
\begin{barticle}
\bauthor{\bsnm{Sen}, \binits{M De la}}:
\batitle{The generalized Beverton-Holt equation and the control of populations}.
\bjtitle{Applied Mathematical Modeling}
\bvolume{32}, 
\bfpage{2312}--\blpage{2328}
(\byear{2008}) 
\end{barticle}
\endbibitem


\bibitem{ulam}
\begin{barticle}
\bauthor{\bsnm{Ulam}, \binits{S M}}:
\batitle{A Collection of Mathematical Problems}.
\bjtitle{Interscience Publ., New York}
(\byear{1960})
\end{barticle}
\endbibitem


\bibitem{XB}
\begin{barticle}
\bauthor{\bsnm{Xu}, \binits{B}},
\bauthor{\bsnm{Brzd\c{e}k}, \binits{J}}:
\batitle{Hyers-Ulam stability of a system of first order linear recurrences with constant coefficients}.
\bjtitle{Discrete Dyn. Nat. Soc.}
\bparticle{ Article ID 269356}
(\byear{2015})
\end{barticle}
\endbibitem


\end{thebibliography}
\end{document}